\documentclass[reqno]{amsart}

\usepackage{amssymb,amsthm,amsmath,amscd}

\newtheorem{theorem}{Theorem}
\newtheorem{lemma}{Lemma}
\newtheorem{corollary}{Corollary}

\theoremstyle{definition}
\newtheorem{example}{Example}
\newtheorem{remark}{Remark}

\numberwithin{theorem}{section}
\numberwithin{lemma}{section}
\numberwithin{corollary}{section}
\numberwithin{example}{section}
\numberwithin{remark}{section}

\def\Var{\mathop{\fam 0 Var}\nolimits}
\def\As{\mathop{\fam 0 As}\nolimits}
\def\Alg{\mathop{\fam 0 Alg}\nolimits}
\def\Lie{\mathop{\fam 0 Lie}\nolimits}
\def\Com{\mathop{\fam 0 Com}\nolimits}
\def\Perm{\mathop{\fam 0 Perm}\nolimits}
\def\ComTriAs{\mathop{\fam 0 ComTriAs}\nolimits}
\def\di{\text{\rm di-}}
\def\tri{\text{\rm tri-}}

\let\<\langle
\let\>\rangle

\title{Gr\"obner---Shirshov bases for replicated algebras}
\author[P.~S.~Kolesnikov]{P.S. Kolesnikov}

\address{Sobolev Institute of Mathematics,\\ Akad. Koptyug prosp., 4\\ 630090 Novosibirsk, Russia}
\email{pavelsk@math.nsc.ru}

\thanks{The research is supported by RSF (project 14-21-00065).}


\begin{document}

\begin{abstract}
 We establish a universal approach to solution of the word problem in the varieties of di- and tri-algebras. 
 This approach, for example, allows to apply Gr\"obner---Shirshov bases method for Lie algebras to solve the
 ideal membership problem in free Leibniz algebras (Lie di-algebras). As another application, we prove an analogue 
 of the Poincar\'e---Birkhoff---Witt Theorem for universal enveloping associative tri-algebra of a Lie tri-algebra
 (CTD$^!$-algebra).
\end{abstract}

\maketitle

\section{Introduction}

Gr\"obner bases theory is known as an effective computational technique in 
commutative algebra and related areas. Various questions in mathematics may be reduced 
to the {\em ideal membership problem} (or {\em word problem}): 
given a set $S$ of (commutative) polynomials, whether a given 
polynomial $f$ belongs to the ideal generated by $S$. 
In non-commutative (or even non-associative) settings, the same problem appears 
mainly in theoretical studies rather than in computational context,
see the review \cite{BokChen13}. 
The classical example is given by the Poincar\'e---Birkhoff---Witt Theorem for Lie algebras 
and its analogues for other classes of algebras (see \cite{MikhShest14}). 

It worths mentioning that Gr\"obner bases in commutative algebras \cite{Buch70} 
appeared simultaneously with standard bases (now called Gr\"obner---Shirshov bases, GSB) 
in Lie algebras \cite{Shir62}. In the last years, GSB theories have been 
established for various classes of non-associative algebras, 
including associative di-algebras \cite{BCL10, ZhangChen17}. 
The latter were introduced in \cite{LodayPir93}
as ``envelopes'' of Leibniz algebras.

Let us sketch what is a GSB theory for a given variety $\Var $ of (linear) algebras over a field~$\Bbbk $.
It usually includes the following components.
\begin{itemize}
 \item Description of the free algebra $\Var\<X\>$ generated by a set~$X$. 
    Elements of $\Var\<X\>$ are called {\em polynomials}, they are linear combinations of {\em monomials} (normal words)
    that form a linear basis of $\Var\<X\>$.
 \item Linear order on monomials compatible with algebraic operations.
 \item Elimination procedure of the leading word $\bar f$ of a polynomial $f$ in a monomial $u$.
 \item Definition of {\em compositions} and the notion of {\em triviality} 
    of a composition modulo a given set of polynomials.
 \item Composition-Diamond Lemma.
\end{itemize}
If $S$ is a set of monic polynomials (principal word appears with identity coefficient) 
such that every composition of its polynomials is trivial modulo $S$
then $S$ is said to be a GSB in $\Var\<X\>$. Principal point of GSB theory, the Composition-Diamond Lemma (CD-Lemma),
usually has the following form: $S$ is a GSB if and only if for 
every element $f$ in the ideal generated by $S$ 
there exists $g\in S$ such that $\bar f$ admits elimination of $\bar g$.
In particular, if $S$ is a GSB in $\Var\<X\>$ then the images 
of {\em $S$-reduced} monomials (those that do not admit elimination of $\bar g$, $g\in S$) 
form a linear basis of the quotient algebra 
$\Var\<X\mid S\>$ generated by $X$ with defining relations~$S$.

In this paper, we propose a general approach to the GSB theory for a class of varieties obtained by 
{\em replication procedures}. The latter, from the categorical point of view, may be easily 
explained in terms of operads. If $\Var $ is the variety governed by an operad $\mathcal P$ (see \cite{GK94}) 
then replicated varieties $\di\Var $ and $\tri\Var $ are governed by Hadamard the products 
$\Perm\otimes \mathcal P$ and $\ComTriAs\otimes \mathcal P$, respectively. Here $\Perm $ and 
$\ComTriAs$ are the operads corresponding to the varieties of Perm-algebras and commutative tri-algebras
introduced in \cite{Chapoton01} and \cite{Vallette_2007}, respectively (see \cite{GubKol2014} for more details).

In particular, for the varieties $\As$ and $\Lie $ of associative and Lie algebras, $\di\As $ and $\di\Lie$ 
coincide with the varieties of associative di-algebras and Leibniz algebras, respectively. Leibniz algebras are known 
as the most common ``non-commutative'' generalization of Lie algebras.
Note that for associative di-algebras GSB theory has been constructed in 
\cite{BCL10} and \cite{ZhangChen17}.  However, there is no GSB theory 
for Leibniz algebras as well as for tri-algebras (i.e., systems from $\tri\Var$). 
This paper aims at filling this gap.

Suppose GSB theory is known for a variety $\Var $. Then we explicitly construct a ``$\Var$-envelope'' for 
free algebras $\di\Var \<X\>$ and $\tri\Var\<X\>$ such that solution of the ideal membership problem 
in this $\Var $-envelope induces a solution of the same problem in $\di\Var \<X\>$ or $\tri\Var\<X\>$.
This approach differs from the usual one described above, but it is easier in applications.   
We will apply the technique developed to study universal enveloping Lie tri-algebras of Lie algebras 
and associative enveloping tri-algebras of Lie tri-algebras.

\section{Replication of varieties}

Let $(\Sigma , \nu)$ be a {\em language}, i.e., a set of operations together with 
arity function $\nu: \Sigma \to \mathbb Z_+$. 
A {\em $\Sigma $-algebra\/} is a linear space $A$ equipped with 
polylinear operations $f: A^{\otimes n}\to A$, $f\in \Sigma $, $n=\nu(f)$. 
Denote by $\Alg = \Alg_{\Sigma }$ the class of all $\Sigma $-algebras, and 
let $\Alg \<X\>$ stand for the free $\Sigma $-algebra generated by a set~$X$. 
(For example, if $\Sigma $ consists of one binary operation then $\Alg\<X\>$ 
is the magmatic algebra.)

For a given language $(\Sigma , \nu)$, define two {\em replicated} languages 
$(\Sigma^{(2)}, \nu^{(2)})$ and $(\Sigma ^{(3)}, \nu^{(3)})$ as follows:
\[
\begin{gathered}
\Sigma^{(2)} = \{f_i \mid f\in \Sigma, i=1,\dots, \nu(f) \},\quad \nu^{(2)}(f_i)=\nu(f); \\  
\Sigma^{(3)} = \{f_H \mid f\in \Sigma, \varnothing\ne H\subseteq \{1,\dots, \nu(f)\} \},
             \quad \nu^{(3)}(f_H)=\nu(f).
\end{gathered}
\]
Denote $\Alg^{(2)}=\Alg_{\Sigma^{(2)}}$ and $\Alg^{(3)}=\Alg_{\Sigma^{(3)}}$.

Note that $\Sigma^{(2)}$ may be considered as a subset of 
$\Sigma^{(3)}$ via $f_i=f_{\{i\}}$. 
This is why we will later deal with $\Sigma^{(3)}$ assuming the same statements (in a more simple form)
 hold for $\Sigma ^{(2)}$. 

Suppose an element $\Phi \in \Alg \<X\>$ is polylinear with respect to $x_1,\dots , x_n \in X$. 
Then for every nonempty subset $H\subset \{1,\dots, n\}$
we may define $\Phi_H\in \Alg^{(3)} \<X\>$ in the following way (see \cite{GubKol2014} for more details). 
Every monomial summand of $\Phi$ may be naturally considered as a rooted tree whose leaves are labelled by 
variables $x_1,\dots, x_n$ and nodes are labelled by symbols of operations from $\Sigma $. 
Every node labelled by $f\in \Sigma $ has one ``input'' and $\nu(f)$ ``outputs'', each output is attached to a subtree. 
Let us emphasize leaves $x_i$ for $i\in H$ and change the labels of nodes 
by the following rule. For every node labelled by a symbol $f\in \Sigma $
consider the subset $S$ of $\{1,\dots, \nu(f)\}$ which consists of those output numbers that are attached 
to subtrees containing emphasized leaves.
If $S$ is nonempty then replace $f$ with $f_S$, if $S$ is empty then replace $f$ with $f_{\{1\}}$.
Transforming every polylinear monomial of $\Phi $ in this way, we obtain $\Phi_H$.
(For $\Sigma^{(2)}$, it is enough to consider $|H|=1$.)

Suppose $\Var $ is a variety of $\Sigma $-algebras defined by a collection of polylinear identities
$S(\Var )\subset \Alg \<X\>$, $X=\{x_1,x_2,\dots \}$.
The following statement may be interpreted as a definition of what is the variety $\tri\Var$. 

\begin{theorem}[\cite{GubKol2014}]
A $\Sigma ^{(3)}$-algebra belongs to the class $\tri\Var $
if and only if it satisfies the following identities in $\Alg^{(3)}\<X\>$:
\begin{multline}\label{eq:Zero-Identities}
 f_H(x_1,\dots , x_{i-1}, g_S(x_i,\dots, x_{i+m-1}), x_{i+m}, \dots, x_{n+m-1}) \\
=
 f_H(x_1,\dots , x_{i-1}, g_Q(x_i,\dots, x_{i+m-1}), x_{i+m}, \dots, x_{n+m-1}),
\end{multline}
where 
$f,g\in \Sigma $, $\nu(f)=n$, $\nu(g)=m$, $H\subseteq \{1,\dots, n\}$, 
$S,Q\subseteq \{1,\dots, m\}$, $i\notin H$, 
and
\begin{equation}\label{eq:triAlgebra}
 \Phi_H(x_1,\dots , x_n) = 0, 
\end{equation}
where $\Phi \in S(\Var)$, $\deg \Phi=n$, $H\subseteq \{1,\dots, n\}$.
\end{theorem}

\begin{example}
Let $\Sigma $ consist of one binary product $\mu(a ,b ) = [a b]$. Consider 
$\Var = \Lie $ ($\mathrm{char}\,\Bbbk \ne 2$) with defining identities 
\[
 [x_1x_2]+[x_2x_1]=0, \quad [[x_1x_2]x_3] + [[x_2x_3]x_1]+[[x_3x_1]x_2]=0. 
\]
Then $\Sigma^{(3)}$ consists of three binary operations 
\[
 \mu_{\{1\}}(a,b) = [a \dashv b], \quad 
 \mu_{\{2\}}(a,b) = [a \vdash b], \quad 
 \mu_{\{1,2\}}(a,b) = [a \perp b].
\]
The family of defining identities of the variety
$\tri\Lie $ of {\em Lie tri-agebras}
contains replicated skew-symmetry 
\[
 [x_1\vdash x_2]+[x_2\dashv x_1]=0, \quad [x_1\perp x_2]+[x_2\perp x_1] =0.
\]
This allows replace $[a\dashv b]$ with $-[b\vdash a]$ and express all defining relations 
in terms of $[\cdot \vdash \cdot ]$ and $[\cdot \perp \cdot]$. 
It is easy to see that \eqref{eq:Zero-Identities} turn into 
\begin{gather}
 {}
[[x_1\vdash x_2]\vdash x_3]=-[[x_2\vdash x_1]\vdash x_3], \label{eq:triLieI} \\
 [[x_1\perp x_2]\vdash x_3]=[[x_1\vdash x_2]\vdash x_3], \label{eq:triLie2}
\end{gather}
and the replication of Jacobi identity leads (up to equivalence) to
the following three relations:
\begin{gather}
{}
 [[x_1\vdash x_2]\vdash x_3] - [x_1\vdash [x_2\vdash x_3]]+[x_2\vdash [x_1 \vdash x_3]]=0, \label{eq:triLie-3} \\
 [[x_1\vdash x_2]\perp x_3] - [x_1\vdash [x_2\perp x_3]]-[[x_1\vdash x_3]\perp x_2]=0, \label{eq:triLie-4} \\
 [[x_1\perp x_2]\perp x_3] + [[x_2\perp x_3]\perp x_1]+[[x_3\perp x_1]\perp x_2]=0. \label{eq:triLie-5}
\end{gather}
Note that \eqref{eq:triLie-3} is the left Leibniz identity, \eqref{eq:triLieI} easily follows from 
\eqref{eq:triLie-3}. Therefore, a Lie tri-algebra may be considered as a linear space with two operations 
$[\cdot \vdash \cdot]$ and $[\cdot \perp \cdot ]$ such that the first one satisfies the left Leibniz identity, 
the second one is Lie, and \eqref{eq:triLie2}, \eqref{eq:triLie-4} hold. 
\end{example}

\begin{remark}
 The operad governing the variety $\tri\Lie$ is Koszul dual to the operad CTD  governing 
 the variety of commutative tridendriform algebras introduced 
 in \cite{Loday2007}. This is a particular case of a general relation 
 between di- or tri- algebras and their dendriform counterparts \cite{GubKol2013}.
 In \cite{Zinb}, $\tri\Lie $ is stated as CTD$^!$. 
\end{remark}

In a similar way, one may construct the defining identities of the variety $\tri\As$: the latter coincides
with the variety of triassociative algebras introduced in \cite{LodayRonco2004}.

\begin{example}\label{exmp:C_2}
Let $C_2$ be the 2-dimensional space with a basis
$\{e_1,e_2\}$ equipped with binary operations
\[
e_i\perp e_i = e_i, \quad 
e_1\vdash e_1 = e_1\dashv e_1 = e_1, \quad
e_1\vdash e_2 = e_2\dashv e_1 = e_2,
\]
other products are zero. It is easy to check that $C_2\in \tri\Com $, 
where $\Com $ is the variety of associative and commutative algebras.
\end{example}

\section{Construction of free di- and tri-algebras}\label{sec3}

In this section, we present simple construction of the free algebra 
$\tri\Var\<X\>$ generated by a given set $X$
in the variety $\tri\Var $. The same construction works for $\di\Var$
after obvious simplifications. 
To make the results more readable, we restrict to the case 
when $\Sigma $ consists of one binary product since this is the most common 
case in practice. However, there are no obstacles to the transfer 
of the following considerations to an arbitrary language.

Given a set $X$, denote by $\dot X$ the copy of $X$: 
\[
 \dot X = \{\dot x \mid x\in X \}.
\]
Denote by $F$ the free algebra $\Var \<X\cup \dot X\>$ in the variety $\Var $. 
There exists unique homomorphism $\varphi : F\to \Var\<X\>\subset F$ determined by 
$x\mapsto x$, $\dot x\mapsto x$, $x\in X$. 

Let us define three binary operations on the space $F$ as follows:
\begin{equation}\label{eq:OperAve}
 f\vdash g = \varphi(f)g,\quad f\dashv g = f\varphi(g), \quad f\perp g = fg, 
\end{equation}
$f,g\in F$. Denote the system obtained by $F^{(3)}$.

\begin{lemma}\label{lem:F3-ave}
Algebra $F^{(3)}$ belongs to $\tri\Var $. 
\end{lemma}

\begin{proof}
Obviously, $\varphi^2=\varphi $, so 
\begin{equation}\label{eq:Averaging}
 \varphi(\varphi(f)g) = \varphi(f\varphi(g)) = \varphi(f)\varphi(g)=\varphi(fg),\quad f,g\in F.
\end{equation}
Relation \eqref{eq:Averaging} means that $\varphi $
is a homomorphic {\em averaging operator\/} on~$F$. It is well known (see \cite[Theorem 2.13]{GubKol2014})
that in this case an algebra from $\Var $ relative to operations \eqref{eq:OperAve} 
belongs to $\tri\Var $.
\end{proof}

For a monomial $u$ in $F$, denote by $\deg_{\dot X} u$ the degree of $u$ with respect to 
variables from $\dot X$.

\begin{lemma}\label{lem:F3-subalg}
The subalgebra $V$ of $F^{(3)}$ generated by $\dot X$ 
coincides with the subspace $W$ of $F$ spanned by all monomials $u$ 
such that $\deg_{\dot X} u>0$.
\end{lemma}

\begin{proof}
It is clear that $V\subseteq W$. Indeed,  
consider $u\vdash v$, $u\dashv v$, $u\perp v$ for $u,v\in V$. 
Inductive arguments allow to assume $u,v\in W$ and thus
all three products also belong to $W$ by the definition of $F^{(3)}$. 
The converse embedding $W\subseteq V$ is proved  analogously. 
\end{proof}

It is easy to see that the space $V$ from Lemma~\ref{lem:F3-subalg} coincides with 
the ideal of $F$ generated by $\dot X$. 

\begin{theorem}\label{thm:FreeTrialg}
The subalgebra $V$ of $F^{(3)}$ is isomorphic to 
the free algebra in the variety $\tri\Var $ generated by $X$.
\end{theorem}

\begin{proof}
It is enough to prove universal property of $V$ in the class $\tri\Var$. 
Suppose $A$ is an arbitrary tri-algebra in $\tri\Var$, and let $\alpha :X\to A$
be an arbitrary map. Our aim is to construct a homomorphism 
of tri-algebras $\chi : V\to A$ such that $\chi(\dot x)=\alpha(x)$ for all $x\in X$.

Recall the following construction (proposed in \cite{GubKol2013}). 
The subspace 
\[
A_0=\mathrm{span}\,\{a\vdash b-a\dashv b,\, a\vdash b-a\perp b\mid a,b\in A \} 
\]
is an ideal of $A$. The quotient $\bar A = A/A_0$ carries
a natural structure of an algebra from $\Var $ given by 
$\bar a \bar b = \overline {a\perp b}$.
Consider the formal direct sum $\hat A = \bar A\oplus A$ 
equipped with one well-defined product
\[
\bar a b=a\vdash b,\quad 
a \bar b= a \dashv b,\quad 
\bar a \bar b =\overline{a\perp b},\quad 
a b=a\perp b,
\]
for
$\bar a, \bar b \in \bar A$, 
where $\bar c = c+A_0\in \bar A$, $c\in A$.
Then $\hat A\in \Var $.

Another important fact on $\Var $ and $\tri\Var$ was established in \cite{GubKol2014}.
For every commutative tri-algebra $C\in \tri\Com$ and for every algebra $B\in \Var $
the linear space $C\otimes B$ equipped with 
\[
 (p\otimes a)*(q\otimes b) = (p*q)\otimes ab,\quad {*}\in \{\vdash, \dashv, \perp\}, \ 
 a,b\in B,\ p,q\in C,
\]
is a tri-algebra in the variety $\tri\Var $.
There is a natural relation between $A\in \tri\Var $ and $\hat A\in \Var$. 
For the tri-algebra $C_2$ from Example~\ref{exmp:C_2} we have an embedding 
of tri-algebras $\iota : A \to C_2\otimes \hat A$ given by
\[
 \iota : a\mapsto e_1\otimes \bar a + e_2\otimes a ,\quad a\in A.
\]

Now, construct $\hat \alpha : X\cup\dot X \to \hat A$ as
\[
 \hat\alpha(x) = \overline{\alpha x} \in \bar A,\quad 
 \hat\alpha (\dot x) = \alpha (x) \in A, 
\]
for $x\in X$. 
The map $\hat \alpha $ induces a homomorphism 
of algebras $\hat \psi : F \to \hat A$. 
Finally, define 
\[
 \psi : F \to C_2\otimes \hat A
\]
by 
\begin{equation}
 \psi(f) = e_1\otimes \hat \psi (\varphi (f)) + e_2\otimes \hat\psi(f), \quad f\in F, 
\end{equation}
Let us show that $\psi $ is a homomorphism of tri-algebras. For every $f,g\in F$, ${*}\in \{{\vdash}, {\dashv}, {\perp}\}$
\begin{multline}\label{eq:TriProdCases}
 \psi(f)*\psi(g) 
 = (e_1\otimes \hat \psi(\varphi (f)) + e_2\otimes \hat\psi(f))*(e_1\otimes \hat \psi(\varphi (g)) + e_2\otimes \hat\psi(g))\\
 =(e_1*e_1)\otimes \hat \psi(\varphi (f)\varphi (g)) 
  + (e_1*e_2)\otimes \hat \psi(\varphi (f)g)
  +(e_2*e_1)\otimes \hat\psi(f\varphi (g))
  +(e_2*e_2)\otimes \hat\psi(fg) \\
 =
 \begin{cases}
  e_1\otimes \hat\psi(\varphi(fg)) + e_2\otimes \hat\psi(\varphi(f)g), & {*}={\vdash} , \\
  e_1\otimes \hat\psi(\varphi(fg)) + e_2\otimes \hat\psi(f\varphi (g)), & {*}={\dashv} , \\
  e_1\otimes \hat\psi(\varphi(fg)) + e_2\otimes \hat\psi(fg), & {*}={\perp} .
 \end{cases}
\end{multline}
On the other hand, it is straightforward to compute $\psi(f*g)$. Since $\varphi $ is a homomorphic averaging 
operator, the results coincide with those in \eqref{eq:TriProdCases}.

It is easy to see from the definition that $\psi(\dot x) = \iota (\alpha (x))$. 
Therefore, $\psi(V) \subseteq \iota (A)\subseteq C_2\otimes \hat A$ by Lemma \ref{lem:F3-subalg}.
Finally, the desired homomorphism $\chi $ may be constructed as 
\[
 \chi = \iota^{-1}\circ \psi|_V : V\to A. 
\]
In other words, the diagram
\[
\begin{CD}
\dot X @>\subseteq >> V @>\subseteq >> F^{(3)} \\
@V\varphi VV @VV\chi V @VV\psi V \\
X @>\alpha >> A @>\iota >> C_2\otimes \hat A
\end{CD}
\]
is commuting.
\end{proof}

\begin{remark}\label{rem:FreeDialg}
 All constructions of this section make sense for di-algebras. It is enough to 
 consider only operations $\vdash $ and $\dashv $ on $F$ given by the same rules. 
 The role of $V\subseteq F$ is played by the subspace of polynomials linear 
 in~$\dot X$. 
\end{remark}

\begin{example}\label{exmp:FreeLeibniz}
 Let $\Var = \Lie $, $\di\Lie $ is the variety of Leibniz algebras. Then 
 \[
  \di\Lie\<X\> \simeq V \subseteq \Lie \<X\cup \dot X\>.
 \]
Suppose $X$ is linearly ordered; let us extend  the order to $X\cup \dot X$
in the natural way:
\[
 x>y \Rightarrow \dot x>\dot y,\quad \dot x>y
\]
for all $x,y\in X$.
It is easy to see that all words of the form
\[
 [\dots [[\dot x_1x_2]x_3]\dots x_n]
\]
(with left-justified bracketing) are linearly independent in $F$ since so are their images in $U(F)\simeq \As\<X\cup \dot X\>$.
These words correspond to 
\begin{equation}\label{eq:RightLeibniz}
 [\dots [[x_1\dashv x_2]\dashv x_3]\dashv \dots\dashv  x_n] \in \di\Lie \<X\>, 
\end{equation}
where $[\cdot \dashv \cdot ]$ satisfies (right) Leibniz identity
\[
 [x\dashv [y\dashv z]] = [[x\dashv y]\dashv z] - [[x\dashv z]\dashv y] .
\]
Obviously, every Leibniz polynomial may be rewritten as a linear combination of monomials \eqref{eq:RightLeibniz}, 
so the latter form a linear basis of $\di\Lie\<X\>$ \cite{Loday93, Loday2001}.
\end{example}

\section{Ideal membership problem}\label{sec4}

As above, let $\Var $ be a variety of algebras (with one binary product) 
defined by polylinear identities, $\di\Var $ and $\tri\Var $
are the corresponding varieties of di- and tri-algebras. 

Suppose $X$ is a nonempty set of generators. 
By Theorem \ref{thm:FreeTrialg} and Remark \ref{rem:FreeDialg}, 
free systems $\tri\Var\<X\>$ and $\di\Var \<X\>$
may be considered as subspaces of $F=\Var \<X\cup \dot X\>$.
As in Section \ref{sec3}, let us denote these subspaces by $V$.
We will consider the case of tri-algebras in details. 

For every $S\subseteq V\subseteq F^{(3)}$ denote by $(S)^{(3)}$ the ideal of $V$
generated by $S$. 
Every tri-algebra $A\in \tri\Var $ may be presented 
by generators and defining relations as 
$A\simeq \tri\Var\<X\mid S\> \simeq  V / (S)^{(3)}$
for appropriate $X$ and $S$.
In order to understand the structure of $A$ we have to know how to decide whether 
a given $f\in V$ belongs to $(S)^{(3)}$. This kind of problems 
is the main target of the Gr\"obner---Shirshov basis (GSB) method. 
In order to translate GSB theory from the class $\Var $ to $\tri\Var $ 
(and to $\di\Var$) we need the following

\begin{theorem}\label{thm:IdealsReplication}
Let $S\subset V \subset F^{(3)}$.
Then 
\[
 (S)^{(3)} = (S\cup \varphi(S)) \cap V, 
\]
where $(P)$ stands for the ideal of $F$ generated by its subset $P$.
\end{theorem}

Here $\varphi $ is the endomorphism of $F$ defined in Section~\ref{sec3}.

\begin{proof}
Denote $I=(S\cup \varphi (S))$. Obviously, 
\[
 I = \bigcup\limits_{s\ge 0} I_s, \quad I_0\subseteq I_1\subseteq \dots, 
\]
where 
$I_0 = \mathrm{span}\,(S\cup \varphi(S))$,
\[
 I_{s+1} = I_s + FI_s + I_sF, \quad s\ge 0.
\]

Similarly, for $J = (S)^{(3)}$ we have 
\[
 J = \bigcup\limits_{s\ge 0} J_s, \quad J_0\subseteq J_1\subseteq \dots, 
\]
where 
$J_0 = \mathrm{span}\,(S)$,
\[
 J_{s+1} = J_s + V\vdash J_s + J_s\dashv V + V\dashv J_s + J_s\vdash V + J_s\perp V + V\perp J_s, \quad s\ge 0.
\]

Since $S\subset V$, 
$\varphi(S) \subset \Var \<X\>$, and $\Var\<X\>\cap V = 0$, we have  
$I_0\cap V = J_0$.
Moreover, $I_0 = J_0+ I_0'$, where $I_0'=\mathrm{span}\,\varphi(S)=\varphi(J_0)\subseteq \Var\<X\>$.

Assume $I_s = J_s + \varphi(J_s)$ for some $s\ge 0$.
Note that $F=V+ \Var\<X\>$ and $\Var\<X\> = \varphi(V)$ by Lemma~\ref{lem:F3-subalg}. Then 
\begin{multline}\nonumber
 I_{s+1} = I_s + (V+ \Var\<X\>)I_s + I_s(V+ \Var\<X\>)\\
 = J_s  + VJ_s + V\varphi(J_s) + \Var\<X\>J_s + J_sV + \varphi(J_s)V + J_s\Var\<X\> \\
  + \varphi(J_s)+ \varphi(J_s)\Var\<X\> +\Var\<X\>\varphi(J_s).
\end{multline}
It remains to note that 
\[
\begin{gathered}
 VJ_s = V\perp J_s,\quad V\varphi(J_s) = V\dashv J_s, \quad \Var\<X\>J_s=V\dashv J_s, \\
 J_sV = J_s\perp V,\quad \varphi(J_s)V = J_s\dashv V, \quad J_s\Var\<X\>=J_s\dashv V. \\
\end{gathered}
\]
Hence, 
\[
 I_{s+1} = J_{s+1} + \varphi(J_s)+ \varphi(J_s)\Var\<X\> +\Var\<X\>\varphi(J_s),
\]
but the latter three summands obviously give $\varphi(J_{s+1})$. 

We have proved $I_s = J_s+\varphi(J_s)$ for all $s\ge 0$. Therefore, 
$I_s\cap V = J_s$, and $I\cap V=J$, as required.
\end{proof}

\begin{remark}\label{rem:IdealsReplication}
 For di-algebras, the statement of Theorem~\ref{thm:IdealsReplication} holds true:
 the intersection of $V\simeq \di\Var\<X\>$ with the ideal generated 
 by $S\cup \varphi(S)$ in $F$ is equal to the ideal generated by~$S$ in the di-algebra~$V$.
\end{remark}

Let $S\subset \tri\Var\<X\>$. Then $I=(S\cup \varphi(S))$ is a $\varphi$-invariant ideal of $F $, 
so one may induce tri-algebra structure on $F/I$.

\begin{corollary}\label{cor:TriFactor}
$\tri\Var\<X\mid S\>$ is isomorphic to the subalgebra of $(F/I)^{(3)}$ generated by $\dot X$. 
Similar statement holds for di-algebras.
\end{corollary}

Theorem \ref{thm:IdealsReplication} and Remark \ref{rem:IdealsReplication} provide an easy approach to GSB theory 
for the classes of tri- and di-algebras ($\tri\Var$ and $\di\Var$, respectively)  modulo the analogous theory for 
the variety $\Var $.
In order to find Gr\"obner---Shirshov basis of an ideal generated by $S\subseteq \tri\Var\<X\>$
one should rewrite the relations from $S$ as elements of $F=\Var\<X\cup \dot X\>$, 
and find the GSB $\hat S$ of the ideal in $F$ generated by $S\cup \varphi(S)$. 
(In the case of di-algebras, it is enough to find a part of the latter GSB, namely, 
those polynomials of degree $\le 1$ in $\dot X$.)
To find a linear basis of $\tri\Lie\<X\mid S\>$ one should just 
consider $\hat S$-reduced monomials in $F$ and 
choose those that belong to $V$.

In order to present an example, let us recall the main features of the Gr\"obner---Shirshov bases 
theory for Lie algebras \cite{Shir62} (see also \cite{Bok72}). 

First, suppose $X$ is a linearly ordered set of generators, $X^*$ is the set of all associative words in $X$ equipped with {\em deg-lex order}, 
i.e., two words are first compared by their length and then lexicographically.
The set of associative {\em Lyndon---Shirshov\/} words (LS-words) consists of 
all such words $u$ that for every presentation $u=vw$, $v,w\in X^*$, we have $u>wv$.
Every associative LS-word $u$ has a unique {\em standard\/} bracketing $[u]$ such that 
$[u]=[[v][w]]$, where $w$ is the longest proper LS-suffix of~$u$. 
An associative LS-word with standard bracketing is called a non-associative LS-word;  linear order on such words 
is induced by the deg-lex order on $X^*$.

The set of all non-associative LS-words in the alphabet $X$ is a linear basis of 
the free Lie algebra $\Lie\<X\>$. Given $0\ne f\in \Lie\<X\>$, $\bar f$ denotes its principal non-associative LS-word.

Next, recall the notion of a composition. For every associative LS-words $w$ and $v$ such that $w=uvu'$
for some $u$ and $u'$ (they may be empty) 
there exists unique bracketing $\{u{*}u'\}$ on the word $u{*}u'$ in the alphabet $X\cup \{*\}$
such that 
\[
 \overline{\{u[v]u'\}} = [w]\in \Lie \<X\>.
\]
Suppose $f$ and $g$ are to monic elements from $\Lie\<X\>$, $\bar f=[w]$, $\bar g=[v]$. 
If $ w=uvu'$ as above, then we say that $f$ and $g$ have a {\em composition of inclusion} 
\[
 (f,g)_w = f - \{u g u'\}.
\]
If $w=uu'$ and $v=u'u''$ for some $u,u',u''\in X^*$ then we say that $f$ and $g$ have a {\em composition of intersection}
\[
 (f,g)_{uu'u''} = \{fu''\} - \{ug\}.
\]

Finally, recall the definition of a Gr\"obner---Shirshov basis (GSB) for Lie algebras. 
A set of monic elements $S\subset \Lie \<X\>$ is said to be a GSB in $\Lie \<X\>$ 
if for every $f,g\in S$ every their composition $(f,g)_w$ is {\em trivial}, i.e., may be presented as
\[
 (f,g)_w = \sum\limits_i \alpha_i \{u_i s_i u_i'\}, \quad s_i\in S,\ 
\]
where $\overline{\{u_i s_i u_i'\}}=[u_i\bar s_i u_i']<[w]$.

A non-associative LS-word $[w]$ is said to be {\em $S$-reduced} if $w$
may not be presented as $w=uvu'$, where $[v]=\bar s$ for some $s\in S$.

\begin{theorem}[CD-Lemma, \cite{Bok72, Shir62}]
 For a set of monic elements $S\subseteq \Lie\<X\>$
 denote $(S)$ the ideal generated by~$S$. Then the following conditions 
 are equivalent:
 \begin{enumerate}
  \item $S$ is a GSB in $\Lie\<X\>$; 
  \item $f\in (S)$ implies $\bar f$ is not $S$-reduced;
  \item the images of $S$-reduced words form a linear basis of $\Lie\<X\>/(S)$.
 \end{enumerate}
\end{theorem}

For associative algebras, one may just ``erase brackets'' in all definitions
and statements  \cite{Berg78, Bok76}.

\begin{example}
 Let $X=\{x,y\}$. Consider the Leibniz algebra $L\in \di\Lie $ 
 generated by $X$ with one defining relation 
 \[
  f = [x\dashv y] + [y\dashv x] + y.
 \]
\end{example}

Define the order on $X\cup \dot X$ by $\dot x>\dot y>x>y$.
According to the general scheme, denote $F=\Lie\<X\cup \dot X\>$.
Then $f$ may be interpreted as $[\dot x y]+[\dot yx] +\dot y\in F$, 
and $\varphi(f)=y$.
Obviously, $f$ and $\varphi (f)$ have a composition of inclusion $(f,\varphi(f))_{[\dot xy]} = [\dot yx] +\dot y$, 
and the latter has no more compositions with $\varphi(f)$. 
Therefore, 
\[
 y,\quad  [\dot yx] +\dot y
\]
is a GSB in $F$. The linear basis of $\di\Lie \<x,y \mid f \>$ consists of all those non-associative 
Lyndon---Shirshov words in $\Lie \<\dot x, \dot y, x, y \>$ that are linear in $\dot X$
that do not contain $y$ or $\dot y x$ as (associative) subwords. Obviously, these are 
\[
 \dot y,\quad [\dots [[\dot x x]x]\dots x].
\]
Hence, the linear basis of $L$ consists of 
\[
 y, \quad [\dots [[ x\dashv  x]\dashv x]\dashv \dots\dashv x ].
\]

\section{Applications}\label{sec5}

Gr\"obner bases in commutative algebra are known as an efficient tool for solving computational problems. 
In non-commutative (and non-associative) settings, GSB technique is mainly used for solving theoretical 
problems. A wide family of such problems is related with the structure of universal envelopes. 

Namely, suppose $\Var_1$ and $\Var_2$ are two varieties of algebras, and let
$\omega $ be a functor from $\Var_1$ to $\Var_2$ which turns $A\in \Var_1$ into $A^{(\omega )}\in \Var_2$, 
where $A^{(\omega )}$ is the same linear space as $A$ equipped with algebraic operations
expressed in terms of operations in $\Var_1$. 
(In terms of operads, this exactly means that the functor $\omega $ is induced by a morphism 
of corresponding operads, see, e.g., \cite{GK94}.
For example, the well-known functor ${-}:\As\to \Lie$ turns an associative algebra $A$ into Lie algebra $A^{(-)}$ 
with new operation $[ab]=ab-ba$, $a,b\in A$.)
Then for every $B\in \Var_2$ there exists unique (up to isomorphism) universal enveloping algebra 
$U_\omega (B)\in \Var_1$. The most natural way to construct $U_\omega (B)$ is to consider a linear basis $X$
of $B$ and express the multiplication table of $B$ as a system $S$ of defining relations in $\Var_1\<X\>$.
Then $U_\omega (B)\simeq \Var_1\<X\mid S\>$. In this section, we consider two such functors
on the varieties $\tri\Lie $ and $\tri\As$, and determine 
the structure of corresponding universal envelopes:
\begin{itemize}
 \item forgetful functor $\perp: \tri\Lie \to \Lie$, $[ab]=[a\perp b]$;
 \item tri-commutator functor $-: \tri\As \to \tri\Lie $, where 
 \begin{equation}\label{eq:triCommutator}
[a\vdash b]=a\vdash b-b\dashv a,\quad [a\dashv b]=a\dashv b-b\vdash a,\quad [a\perp b] = a\perp b - b\perp a.  
 \end{equation}
\end{itemize}

Let $L$ be a Lie algebra with linear basis $X$ and multiplication table $\mu : X\times X \to \Bbbk X$, 
$\mu(x,y)$ is a linear form in $X$ for every $x,y\in X$. Assume $X$ is linearly ordered.
Then $U_{\perp}(L)\simeq \tri\Lie \< X \mid S\>$,
where $S=\{ [x\perp y] - \mu(x,y)\mid  x,y\in X, x>y \}$. 
According to the scheme described in Section~\ref{sec4}, we have to consider 
\[
 U = \Lie \<X\cup \dot X \mid [\dot x \dot y]-\dot \mu(x,y), [xy]-\mu(x,y), x,y\in X, x>y \>,
\]
where $\dot\mu(x,y) = \sum_i \alpha_i \dot z_i$
for 
$\mu(x,y) = \sum_i \alpha_i z_i$.

Note that polynomials from $S=\{[\dot x \dot y]-\dot \mu(x,y) \mid x,y\in X, x>y \}$ 
do not have compositions, so $S$
is a GSB. The same applies to $\varphi (S)=\{[xy]-\mu(x,y) \mid x,y\in X, x>y \}$.
Moreover, polynomials from $S$ and $\varphi(S)$
have no compositions since they depend on different variables. Hence, 
$S\cup \varphi(S)$ is a GSB and 
$U$ is isomorphic to the free product $\dot L* L$ \cite{Shir62},
where $\dot L$ is the isomorphic copy of $L$. Corollary \ref{cor:TriFactor} implies

\begin{theorem}\label{thm:Lperp}
 The universal enveloping Lie tri-algebra $U_{\perp}(L)$ of a given Lie algebra $L$ 
 is isomorphic as a linear space to the ideal of $\dot L*L$ generated by~$\dot L$.
\end{theorem}

\begin{corollary}
The pair of varieties $(\tri\Lie, \Lie )$ is a PBW-pair in the sense of \cite{MikhShest14}.
\end{corollary}

Indeed, $L$ embeds into $U_{\perp}(L)$ and there exists a basis of $U_{\perp}(L)$ which does not depend on 
the particular multiplication table of $L$.

Now, let $A$ be an associative tri-algebra with operations $\vdash$, $\dashv $, and $\perp $.
Then $A^{(-)}$ with new operations \eqref{eq:triCommutator} is known to be a Lie tri-algebra. 
Given $L\in \tri\Lie$, denote its universal enveloping associative tri-algebra by $U_-(L)$.
The structure of $U_-(L)$ was studied in \cite{GubKol2014}. Let us show how to apply GSB approach to get 
the same result. A similar computation for di-algebras was performed in \cite{BCL10} in the 
framework of GSB theory for associative di-algebras developed in that paper. Our aim is to show that 
the approach proposed in Section~\ref{sec4} allows to solve such problems with shorter computations.

Suppose $L\in \tri\Lie$, $L_0=\mathrm{span}\,\{[a\vdash b]-[a\dashv b],\, [a\vdash b]-[a\perp b]\mid a,b\in L \}$
as in Section~\ref{sec3}, and let $X$ be a basis of $L$ such that $X=X_0\cup X_1$, $X_0$ is a basis of $L_0$.
Denote by $\mu_\vdash$, $\mu_\dashv$, and $\mu_\perp$ the linear forms corresponding to the operations on $L$.
Since $\mu_\vdash(x,y)=-\mu_\dashv(y,x)$ and $\mu_\perp (x,y)=-\mu_\perp(y,x)$, we have to consider 
\[
 U = \As\<X\cup \dot X \mid S \cup \varphi(S) \>, 
\]
where $S=S_\dashv \cup S_\perp$, 
\[
\begin{gathered}
 S_\dashv = \{\dot xy-y\dot x - \dot \mu_\dashv(x,y) \mid x,y\in X\},\\
 S_\perp = \{\dot x\dot y - \dot y\dot x - \dot \mu_\perp(x,y) \mid x,y\in X, x>y \}.
\end{gathered} 
\]
Note that 
\[
 \varphi(S_\dashv) = \{xy-yx - \mu_\dashv(x,y) \mid x,y\in X\} = \bar S_\dashv\cup \bar S_\vdash \cup \bar S_0,
\]
where 
\[
\begin{gathered}
 \bar S_\dashv = \{xy-yx - \mu_\dashv(x,y) \mid x,y\in X, x>y\}, \\
 \bar S_\vdash = \{xy-yx + \mu_\dashv(y,x) \mid x,y\in X, x>y\}, \\
 \bar S_0 = \{\mu_\dashv(x,x) \mid x\in X\},
\end{gathered}
\]
and
\[
 \varphi(S_\perp ) = \{xy - yx - \mu_\perp(x,y) \mid x,y\in X, x>y \}.
\]
Hence, the elements of $\varphi (S)$ have the following compositions of inclusion: 
\[
 \mu_\dashv(x,y)+\mu_\dashv(y,x), \quad \mu_\dashv(x,y) - \mu_\perp(x,y), \quad \mu_\dashv(y,x)+\mu_\perp (x,y), \quad \mu_\dashv(x,x),
\]
where $x>y$. Obviously, the linear space of these compositions coincides with $L_0$, so we may add letters $X_0$ to the defining relations. 
Since $\mu_\dashv (X,X_0)=0$ in $L$ and $\mu_\dashv (X_0,X)\subset L_0$, it is enough to consider the following defining relations in $\As\<X\cup \dot X\>$:
\begin{gather}
 x,\quad x\in X_0; \label{eq:0dotCom} \\
 xy - yx - \mu_\dashv (x,y),\quad x,y\in X_1,\ x>y; \label{eq:00dotCom} \\
 \dot x y - y\dot x - \dot\mu_\dashv (x,y),\quad x\in X,\ y\in X_1; \label{eq:1dotCom} \\
 \dot x\dot y - \dot y \dot x - \dot \mu_\perp(x,y),\quad x,y\in X,\ x>y. \label{eq:2dotCom}
\end{gather}
Let us use the same order on $X\cup \dot X$ as in Example \ref{exmp:FreeLeibniz} along with deg-lex order on $(X\cup\dot X)^*$.

\begin{theorem}\label{thm:TriAs_GSB}
 Relations \eqref{eq:0dotCom}--\eqref{eq:2dotCom} form a GSB in $\As\<X\cup \dot X\>$.
\end{theorem}

\begin{proof}
 Relations \eqref{eq:00dotCom} correspond to the multiplication table of the Lie algebra $\bar L=L/L_0$, 
 so all their compositions of intersection are trivial. The same holds for \eqref{eq:2dotCom}: 
 it corresponds to the Lie algebra $L^{(\perp)}$.
 It remains to compute two families of compositions $(f,g)_w$:
 \begin{enumerate}
  \item $f=yz - zy - \mu_\dashv (y,z)$, $g=\dot x y - y\dot x - \dot\mu_\dashv (x,y)$, $w=\dot x y z$, 
  where $y,z\in Z_1$, $y>z$, $x\in X$;
  \item $\dot x\dot y - \dot y \dot x - \dot \mu_\perp(x,y)$, $g=\dot y z - z\dot y -\dot\mu_\dashv (y,z)$, $w=\dot x\dot y z$, $x,y\in X$, $x>y$, $z\in X_1$.
 \end{enumerate}
Let us compute in details the second one:
\[
 (f,g)_w = fz - \dot x g = -\dot y\dot x z -\dot\mu_\perp(x,y)z +\dot x z \dot y + \dot x\dot \mu_\dashv (y,z). 
\]
For $h,h'\in \As\<X\cup \dot X\>$, let us write
$h\equiv h'$ if $h-h' = \sum\limits_i \alpha_i u_is_iu_i'$, where $s_i$ belong to   \eqref{eq:0dotCom}--\eqref{eq:2dotCom}, $\alpha_i\in \Bbbk $,
and $u_i\bar s_i u_i'<w=\dot x\dot y z$.
Then 
\begin{multline}\label{eq:Composition} 
(f,g)_w 
\equiv 
 -\dot yz\dot x -\dot y\dot\mu_\dashv(x,z) -\dot\mu_\perp (x,y)z +z\dot x\dot y +\dot \mu_\dashv (x,z)\dot y +\dot x\dot\mu_\dashv(y,z) \\
\equiv 
 -\dot\mu_\dashv (y,z)\dot x -\dot y\dot\mu_\dashv(x,z) -\dot\mu_\perp (x,y)z + z\dot\mu_\perp(x,y) +\dot \mu_\dashv (x,z)\dot y +\dot x\dot\mu_\dashv(y,z)\\
 =
 \dot x\dot\mu_\dashv(y,z)-\dot\mu_\dashv (y,z)\dot x 
  +\dot \mu_\dashv (x,z)\dot y -\dot y\dot\mu_\dashv(x,z) 
  + z\dot\mu_\perp(x,y)  -\dot\mu_\perp (x,y)z \\
\equiv 
\dot\mu_\perp(x, \mu_\dashv (y,z)) + \dot\mu_\perp (\mu_\dashv (x,z), y) - \dot\mu_\dashv (\mu_\perp(x,y),z)
\end{multline}
The right-hand side of \eqref{eq:Composition} is equal to 
$[x\perp [y\dashv z] ]+[[x\dashv z ]\perp y] - [[x\perp y]\dashv z]\in L$
which is zero in every Lie tri-algebra.
\end{proof}

\begin{corollary}[\cite{GubKol2014}]\label{eq:TUbasis}
 As a linear space, $U_-(L)$ is isomorphic to $U(\bar L)\otimes U_0(L^{(\perp)} )$, 
 where $U(\bar L)$ is the usual universal enveloping associative algebra of the Lie algebra $\bar L=L/L_0$ 
 and $U_0(L^{(\perp )})$ is the augmentation ideal of $U(L^{(\perp )})$.
\end{corollary}

\begin{proof}
Linear basis of $U_-(L)$ consists of those associative words of degree $>0$ in $\dot X$ that are reduced 
relative to \eqref{eq:0dotCom}--\eqref{eq:2dotCom}:
\[
 x_1\dots x_k \dot y_1\dots \dot y_m, \quad x_i\in X_1,\ y_j\in X,
\]
where 
$x_1\le \dots \le x_k$, $k\ge 0$, $y_1\le \dots \le y_m$, $m\ge 1$.
\end{proof}

\begin{remark}
For di-algebras, the functor $-:\di\As\to \di\Lie $ and its left adjoint $U_-(\cdot)$ were studied in \cite{AG03, BCL10, Kol2008, Loday2001}.
The same result may easily be obtained with our technique by restriction of Theorem~\ref{thm:TriAs_GSB} to $\dot X$-linear relations. 
\end{remark}

\end{document}